\documentclass[12pt]{article}
\usepackage{preprint-layout-light} \usepackage{preprint-notation} \usepackage{doi}
\usepackage{booktabs}
\usepackage{needspace}
\usepackage{placeins}
\usepackage{hyperref}
\hypersetup{hidelinks}

\title{CutIGA with Least Squares Stabilized Nitsche Boundary Conditions:\\ The Biharmonic Problem}
\date{\today}

\author{
Mats G. Larson and
Karl Larsson
}
\date{}

\begin{document}

\maketitle

\begin{abstract}
We propose a symmetric Nitsche method for the biharmonic Dirichlet problem on a smooth domain whose boundary may cut arbitrarily through the computational mesh. Bulk least-squares stabilization in a boundary strip and tangential boundary penalties control the consistency terms through an explicit lifting of the two Dirichlet traces. This gives coercivity on $H^4(\Omega)$ without inverse inequalities and permits a fixed penalty independent of the cut configuration for sufficiently thin strips. Under suitable regularity assumptions, splines of degree $p\geq4$ give energy and $L^2$ error bounds of orders $h^{p-1}$ and $h^{p+1}$. We also analyze a face-stabilized extension to $C^1$ splines of degree $p=2$ and $C^2$ splines of degree $p=3$. Exact-ellipse experiments examine convergence, cut stability, and parameter sensitivity.
\end{abstract}

\section{Introduction}
Robust weak enforcement of boundary conditions is a central issue in unfitted discretizations of higher-order partial differential equations, including models of plates, shells, and thin films. Several boundary traces must be imposed simultaneously, and small intersections between the physical domain and background elements make penalty selection difficult.

\paragraph{Model Problem.}
On a smooth bounded domain $\Omega\subset\IR^d$, we consider
\begin{subequations}\label{eq:model-prob}
\begin{alignat}{2}
\Delta^2u&=f &\qquad&\text{in }\Omega\label{eq:model-prob-a}\\
u&=g_0 &&\text{on }\partial\Omega\label{eq:model-prob-b}\\
\partial_nu&=g_1 &&\text{on }\partial\Omega\label{eq:model-prob-c}
\end{alignat}
\end{subequations}
where $n$ is the outward unit normal. The energy space is $H^2(\Omega)$, and both the value and normal derivative are essential boundary conditions. The corresponding weak form uses the bilinear form $(\Delta u,\Delta v)_\Omega$ with both homogeneous traces imposed on the test functions.

\paragraph{Cut Isogeometric Analysis.}
Computer-aided design (CAD) commonly describes geometry by patches based on non-uniform rational B-splines. Isogeometric analysis uses these representations in finite element analysis and provides approximation spaces with high continuity \cite{CoHu09}. In trimmed geometries, the boundary can cross the computational grid independently of the knot lines. Nitsche's method \cite{nitsche1971} imposes the essential conditions weakly while preserving consistency.

\paragraph{Penalty Selection.}
Insufficient penalties in symmetric Nitsche's method can cause loss of coercivity and accuracy; excessive penalties can cause locking when the discrete space cannot reproduce the prescribed traces, particularly on curved boundaries. Without additional stabilization, arbitrarily small physical portions of cut elements lead to cut-dependent stability requirements. Local generalized eigenvalue problems provide sufficient penalties for second- and fourth-order spline discretizations \cite{EmDo10}, but unfavorable cuts can force large values and impair accuracy \cite{PrLeMa17}.

Ghost-penalty stabilization permits cut-independent boundary penalties, although their sufficient size generally depends on the stabilization weight \cite{BuCl15}. Other approaches include minimal stabilization by local reconstruction of boundary fluxes \cite{BuPuVa20} and penalty-free nonsymmetric Nitsche methods \cite{Bu12}, including for Kirchhoff plates \cite{ScHaHs16}. The latter avoid a boundary penalty by giving up symmetry. Higher-order Dirichlet problems also require balancing different traces: for the biharmonic problem, the boundary value and normal derivative carry different mesh scalings.

\paragraph{Earlier Work.}
Nitsche's formulation is used in discontinuous Galerkin methods \cite{baker1977,wheeler1978,arnold1982}, domain decomposition \cite{BeHaSt2003}, unfitted interface problems \cite{HaHa2002}, and fictitious domain methods \cite{BuCl15}. Least-squares stabilized Nitsche boundary conditions for the Poisson problem were proposed and analyzed in \cite{ElfLarLar19} and subsequently applied to the finite cell method in \cite{LarKolRanLar22}.

For fourth-order problems, \cite{BuHaLa20} combines $C^1$ Bogner--Fox--Schmit elements with face stabilization in an analysis focused on simply supported plates. The discrete-extension framework of \cite{BuHaLa22} includes the clamped biharmonic problem with $C^1$ Hermite splines. The present method instead controls the consistency terms on the full space $H^4(\Omega)$ through bulk and tangential least-squares terms. For the $C^3$ or smoother splines considered in the main analysis, assembly uses only cell and physical-boundary integrals.

\paragraph{Contribution.}
Passing from the Poisson method in \cite{ElfLarLar19} to the biharmonic problem requires simultaneous control of two Dirichlet traces and their different derivative scalings. We construct an explicit Hermite lifting of the boundary value and normal derivative, supported in an interior strip, which converts both consistency terms into volume pairings. Bounds on this lifting and its Laplacian determine sufficient relative weights for the bulk stabilization and all six boundary penalties. The resulting method, called LS Nitsche below, is coercive without inverse inequalities and has energy and $L^2$ error orders $p-1$ and $p+1$ for splines of degree $p\geq4$. The method admits fixed, cut-independent penalties for each choice of the bulk stabilization parameter; the strip must nevertheless be sufficiently thin relative to the boundary geometry, so the guarantee includes a geometry-dependent upper bound on the mesh size when the strip width is proportional to the mesh size. We also analyze LS Nitsche for maximum-continuity splines of degrees $p=2,3$, with one face stabilization term for each degree. Numerical experiments test convergence, stability under changing cuts and refinement, and the accuracy consequences of varying the stabilization parameters.

\paragraph{Scope.}
The analysis covers dimensions $d=2,3$ on smooth exact domains, with sufficiently thin strips and the solution regularity stated in Section~\ref{section:error-est}. It uses exact integration and exact normal and surface differential operators; geometry approximation and quadrature errors are outside these estimates. The numerical evidence concerns quartic and quadratic splines on a two-dimensional exact ellipse. The three-dimensional results and the $C^2$ cubic spline extension are established analytically.

\paragraph{Outline.}
Section~\ref{section:method} defines the method. Section~\ref{section:error-est} gives the analysis and the lower-continuity extension. Section~\ref{section:numerics} presents numerical results, and Section~\ref{section:conclusions} gives conclusions and an outlook. The lifting estimates and face-stabilization proofs are in Appendices~\ref{section:cont-consistency} and~\ref{section:H2-coercivity}; Appendix~\ref{section:numerical-verification} records numerical verification details.

\section{Method}\label{section:method}
\subsection{Domain and Approximation Space}
Let $\Omega\subset\IR^d$, $d=2,3$, be bounded with smooth boundary $\Gamma=\partial\Omega$ and outward unit normal $n$; see Figure~\ref{fig:geometry}(\subref{fig:domain}). We write $(\cdot,\cdot)_\omega$ and $\|\cdot\|_\omega$ for the $L^2$ inner product and norm on $\omega$, and $H^s(\omega)$ for the Sobolev spaces. The notation $\lesssim$ suppresses constants independent of $h$ and the cut configuration. Put $V=H^4(\Omega)$.

\paragraph{Computational Mesh.}
A background mesh $\mcT_h$ of shape-regular elements covers $\Omega$ independently of its boundary. For a set $\omega$, define
\begin{equation}\label{eq:active-mesh}
\mcT_h(\omega)=\{T\in\mcT_h:T\cap\omega\neq\emptyset\}
\end{equation}
Thus $\mcT_h(\Omega)$ is the active mesh, illustrated in Figure~\ref{fig:geometry}(\subref{fig:mesh-a}). The stability argument on $V$ does not require a Cartesian mesh; the spline approximation and numerical experiments use uniform Cartesian grids, with $h$ the cell diameter. For the main method, $V_h$ consists of restrictions to the active mesh of tensor-product B-splines of degree $p\geq4$ and maximum continuity $C^{p-1}$. Their restrictions to $\Omega$ lie in $V$.

\begin{figure}
\centering
\begin{subfigure}[t]{0.36\textwidth}\centering
\includegraphics[height=0.74\linewidth]{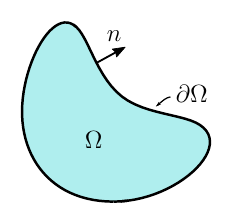}
\subcaption{Domain}
\label{fig:domain}
\end{subfigure}\hspace{0.06\textwidth}
\begin{subfigure}[t]{0.36\textwidth}\centering
\includegraphics[height=0.74\linewidth]{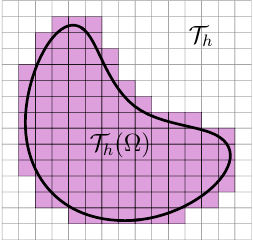}
\subcaption{Active Mesh}
\label{fig:mesh-a}
\end{subfigure}
\par\medskip
\begin{subfigure}[t]{0.36\textwidth}\centering
\includegraphics[height=0.74\linewidth]{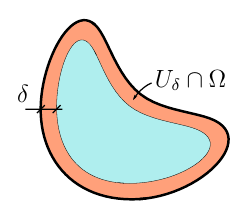}
\subcaption{Stabilization Strip}
\label{fig:U}
\end{subfigure}\hspace{0.06\textwidth}
\begin{subfigure}[t]{0.36\textwidth}\centering
\includegraphics[height=0.74\linewidth]{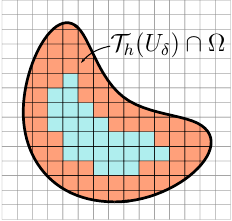}
\subcaption{Mesh-Based Stabilization Strip}
\label{fig:Uh}
\end{subfigure}
\caption{(a) Domain $\Omega$ with smooth boundary $\partial\Omega$ and outward normal $n$. (b) Active mesh $\mcT_h(\Omega)$ of background elements intersecting $\Omega$. (c) Interior stabilization strip $D_\delta=U_\delta\cap\Omega$ of width $\delta$. (d) Alternative mesh-based stabilization strip $D_\delta^h=\mcT_h(U_\delta)\cap\Omega$.}
\label{fig:geometry}
\end{figure}

\paragraph{Boundary Strip.}
Choose a tubular radius $\delta_0>0$ for $\Gamma$ and define
\begin{equation}\label{eq:strip}
U_\delta=\bigcup_{x\in\Gamma}B_\delta(x),
\qquad D_\delta=U_\delta\cap\Omega,
\qquad 0<\delta\leq\delta_0
\end{equation}
The bulk least-squares term is integrated over $D_\delta$, shown in Figure~\ref{fig:geometry}(\subref{fig:U}), and assembled from the regions $T\cap D_\delta$. Its inner boundary is the physical distance surface and is independent of the background mesh. We choose $\delta=\rho h$ with fixed $\rho>0$ for approximation estimates.

\begin{rem}[Mesh-Based Stabilization Strip]
The lifting argument also permits the mesh-based stabilization strip $D_\delta^h=\mcT_h(U_\delta)\cap\Omega$ shown in Figure~\ref{fig:geometry}(\subref{fig:Uh}), understood as the union of the intersecting cells restricted to $\Omega$, as in \cite{ElfLarLar19}. Using this strip avoids constructing the inner boundary of the physical strip. It contains the physical strip but changes with the mesh; the experiments use the boundary strip \eqref{eq:strip} to avoid that additional dependence.
\end{rem}

\subsection{LS Nitsche}
\paragraph{Boundary Penalty.}
Let $\nabla_\Gamma=(I-n\otimes n)\nabla$ and $\Delta_\Gamma=\nabla_\Gamma\cdot\nabla_\Gamma$ be the tangential gradient and Laplace--Beltrami operator. Define the penalty seminorm by its six trace terms:
\begin{align}\label{eq:bh}
\|v\|_{b_h}^2
&=\left(16+\frac1{3\tau}\right)
 \left(\delta^{-3}\|v\|_\Gamma^2+\frac{\delta^{-1}}{30}\|\partial_nv\|_\Gamma^2\right)\\
&\quad+4\left(\delta^{-1}\|\nabla_\Gamma v\|_\Gamma^2
 +\frac{\delta}{30}\|\nabla_\Gamma\partial_nv\|_\Gamma^2\right)\nonumber\\
&\quad+\delta\|\Delta_\Gamma v\|_\Gamma^2
 +\frac{\delta^3}{30}\|\Delta_\Gamma\partial_nv\|_\Gamma^2\nonumber
\end{align}
where $\tau>0$ is a parameter. The numerical weights follow from the lifting estimates used to prove Lemma~\ref{lem:const-bound}. The associated symmetric bilinear form is denoted by $b_h(v,w)$. We write $b_h(g,v)$ for the same pairings with the first argument's value and normal derivative replaced by the prescribed traces $(g_0,g_1)$. It is useful to separate the reciprocal bulk weight:
\begin{equation}\label{eq:boundary-quadratic-norm}
\|v\|_{b_h}^2=B_0(v)+\tau^{-1}B_1(v),
\qquad B_1(v)=\frac13\left(\delta^{-3}\|v\|_\Gamma^2+
 \frac{\delta^{-1}}{30}\|\partial_nv\|_\Gamma^2\right)
\end{equation}
where $B_0$ comprises the terms in \eqref{eq:bh} independent of $\tau$. Thus weakening the bulk stabilization increases the two standard trace penalties in a prescribed balance.

\paragraph{Method.}
Find $u_h\in V_h$ such that
\begin{equation}\label{eq:method}
A_h(u_h,v_h)=L_h(v_h)\qquad v_h\in V_h
\end{equation}
where
\begin{subequations}
\begin{align}
a_h(v,w)&=(\Delta v,\Delta w)_\Omega
 +\tau\delta^4(\Delta^2v,\Delta^2w)_{D_\delta}\label{eq:ah}\\
Q(v,w)&=(\partial_n\Delta v,w)_\Gamma-(\Delta v,\partial_nw)_\Gamma\label{eq:Q}\\
A_h(v,w)&=a_h(v,w)+Q(v,w)+Q(w,v)+\beta b_h(v,w)\label{eq:Ah}\\
L_h(v)&=(f,v)_\Omega+\tau\delta^4(f,\Delta^2v)_{D_\delta}
 +(g_0,\partial_n\Delta v)_\Gamma-(g_1,\Delta v)_\Gamma+\beta b_h(g,v)\label{eq:Lh}
\end{align}
\end{subequations}
Here $\tau>0$ trades weight between bulk stabilization and boundary penalization, while $\beta>0$ is the common boundary penalty parameter. The analysis below gives the sufficient choice $\beta=30$ for sufficiently thin boundary strips.

\begin{rem}[Comparison with Standard Nitsche]\label{rem:standard-nitsche}
Using the same strip-width scaling, write the value and normal-derivative penalties of standard symmetric Nitsche as
\begin{equation}\label{eq:standard-trace-penalties}
\gamma_0\delta^{-3}(v,w)_\Gamma
+\gamma_1\delta^{-1}(\partial_nv,\partial_nw)_\Gamma
\end{equation}
The corresponding terms in $\beta b_h(v,w)$ have the coefficients
\begin{equation}\label{eq:ls-trace-penalty-constants}
\gamma_0=\beta\left(16+\frac{1}{3\tau}\right),
\qquad \gamma_1=\frac{\gamma_0}{30}
\end{equation}
Thus the two penalty constants are tied by $\gamma_0=30\gamma_1$. Decreasing $\tau$ at fixed $\beta$ weakens the bulk stabilization and increases both coefficients. The sufficient choice $\beta=30$ is independent of $\tau$ and, for each fixed $\tau$, gives fixed, cut-independent $\gamma_0$ and $\gamma_1$ for sufficiently thin strips.

LS Nitsche additionally includes the bulk least-squares term and four tangential derivative penalties. Omitting these terms and their load counterparts yields standard symmetric Nitsche with matching $\gamma_0$ and $\gamma_1$; the volume Laplacian form and symmetric consistency terms are the same. The additional terms control the boundary lifting used in Section~\ref{section:error-est}.
\end{rem}

\paragraph{Consistency.}
For a sufficiently regular exact solution, $\Delta^2u=f$ on $D_\delta$ and the two traces agree with the data, so $b_h(u,v)=b_h(g,v)$. Green's formula then gives
\begin{equation}\label{eq:galorth}
A_h(u-u_h,v_h)=0\qquad v_h\in V_h
\end{equation}

\paragraph{Method Extended to Lower-Degree Splines.}
For degrees $p=2,3$, let $V_h$ consist of tensor-product B-splines with maximum continuity $C^{p-1}$ on a uniform Cartesian mesh. Evaluate higher derivatives elementwise and surface derivatives separately on each $\Gamma\cap T$. For $\gamma_p>0$, add the consistent face form
\begin{equation}\label{eq:H2-face-stabilization}
j_h(v,w)=\gamma_p h^{2p-3}\sum_{F\in\mcF_{h,\delta}}
(\ldb\partial_{n_F}^p v\rdb,\ldb\partial_{n_F}^p w\rdb)_{F\cap\Omega}
\end{equation}
where $\mcF_{h,\delta}$ comprises interior active faces intersecting $D_\delta$, $n_F$ is a fixed face normal, and $\ldb q\rdb=q|_{T^+}-q|_{T^-}$ uses that same normal on both sides. This penalizes $\ldb\Delta v\rdb$ for $p=2$ and $\ldb\partial_{n_F}\Delta v\rdb$ for $p=3$; the other face jump in Green's formula vanishes. Set $A_{h,j}=A_h+j_h$. The method is
\begin{equation}\label{eq:H2-method}
A_{h,j}(u_h,v_h)=L_h(v_h)\qquad v_h\in V_h
\end{equation}
The penalized jump vanishes for a smooth exact solution, so the load is unchanged. Section~\ref{section:face-analysis} gives the extension's stability and error estimates.

\section{Analysis}\label{section:error-est}
\subsection{Lifting and Stability}
Define the volume seminorm and energy norm by
\begin{equation}\label{eq:energy-norm}
\|v\|_{a_h}^2=a_h(v,v),\qquad
\tn v\tn_h^2=\|v\|_{a_h}^2+\|v\|_{b_h}^2
\end{equation}

\paragraph{Idea of the Lifting Argument.}
Let $z\in H^2(\Omega)$ be supported in $D_\delta$ with $z=w$ and $\partial_nz=\partial_nw$ on $\Gamma$. Two applications of Green's formula give
\begin{equation}\label{eq:lifting-identity-main}
Q(v,w)=(\Delta^2v,z)_{D_\delta}-(\Delta v,\Delta z)_{D_\delta}
\end{equation}
Consequently,
\begin{equation}\label{eq:lifting-CS-main}
|Q(v,w)|\leq\|v\|_{a_h}
\left(\|\Delta z\|_{D_\delta}^2+\tau^{-1}\delta^{-4}\|z\|_{D_\delta}^2\right)^{1/2}
\end{equation}
The lifting is supported in $D_\delta$ because the bulk least-squares term controls $\Delta^2v$ only there. A cubic Hermite lifting in the normal direction makes the final factor at most $c_\delta\|w\|_{b_h}$. Its derivative and mass estimates determine the six trace weights in \eqref{eq:bh}. In particular, pairing the bulk residual weighted by $\tau\delta^4$ with the lifting mass introduces $\tau^{-1}$ in the two standard trace penalties. No inverse inequality enters this argument.

\begin{lem}[Continuity of the Consistency Form]\label{lem:const-bound}
For $v,w\in V$ and $0<\delta\leq\delta_0$,
\begin{equation}\label{eq:trace}
|Q(v,w)|\leq c_\delta\|v\|_{a_h}\|w\|_{b_h},
\qquad c_\delta^2\leq\tfrac{33}{4}+C\delta
\end{equation}
where $C$ depends only on the boundary geometry and the chosen tubular radius, independently of $\tau$, $h$, and the cut configuration.
\end{lem}
\begin{proof}
Use the Hermite lifting bound \eqref{eq:lifting-bound} from Appendix~\ref{section:hermite-lifting} in \eqref{eq:lifting-CS-main}.
\end{proof}

\begin{ass}[Geometry, Parameters, and Regularity]\label{ass:analysis}
The domain is fixed and smooth, $\tau>0$ and $\rho>0$ are fixed, $\delta=\rho h$, and $\beta\geq30$ is fixed. We take $0<h\leq h_0\leq1$, with $h_0$ sufficiently small that the strip $D_\delta$ lies in the tubular neighborhood $U_{\delta_0}$, $c_\delta^2\leq12$, and the cut-independent trace inequality \eqref{eq:cut-trace-interpolation} holds. The homogeneous clamped problem has the regularity estimate
\begin{equation}\label{eq:dual-regularity}
\|z\|_{H^4(\Omega)}\leq C_{\rm reg}\|r\|_\Omega,
\qquad \Delta^2z=r,\quad z=\partial_nz=0\text{ on }\Gamma
\end{equation}
\end{ass}
The mesh-size restriction depends on the geometry and $\rho$, not on how the boundary cuts the mesh. The sufficient lower bound on $\beta$ is independent of $\tau$, but the error estimates below are stated for fixed $\tau$ and do not assert uniform error constants as $\tau$ varies. The regularity condition is used only for the $L^2$ error estimate.

\begin{lem}[Continuity and Coercivity]\label{lem:continuity-coercivity}
Under the geometric and parameter conditions of Assumption~\ref{ass:analysis},
\begin{alignat}{2}
|A_h(v,w)|&\leq C_A\tn v\tn_h\tn w\tn_h&\qquad &v,w\in V\label{eq:continuity}\\
A_h(v,v)&\geq\tfrac12\tn v\tn_h^2&\qquad &v\in V\label{eq:coercivity}
\end{alignat}
with a cut-independent constant $C_A$.
\end{lem}
\begin{proof}
\proofparagraph{Continuity.}
Lemma~\ref{lem:const-bound} and Cauchy--Schwarz give
\begin{align}
|A_h(v,w)|&\leq\|v\|_{a_h}\|w\|_{a_h}
+c_\delta(\|v\|_{a_h}\|w\|_{b_h}+\|v\|_{b_h}\|w\|_{a_h})
+\beta\|v\|_{b_h}\|w\|_{b_h}\\
&\leq(c_\delta+\max\{1,\beta\})\tn v\tn_h\tn w\tn_h
\end{align}

\proofparagraph{Coercivity.}
Lemma~\ref{lem:const-bound} and Young's inequality imply
\begin{align}
A_h(v,v)&\geq\|v\|_{a_h}^2-2c_\delta\|v\|_{a_h}\|v\|_{b_h}
+\beta\|v\|_{b_h}^2\\
&\geq\tfrac12\|v\|_{a_h}^2+(\beta-2c_\delta^2)\|v\|_{b_h}^2
\geq\tfrac12\tn v\tn_h^2
\end{align}
\end{proof}
Since $c_\delta^2\to33/4$ as $\delta\to0$, the condition $c_\delta^2\leq12$ holds for sufficiently thin strips. Thus $\beta=30$ is a sufficient common penalty for every cut configuration in this regime; the normalized coefficients, rather than the numerical value of $\beta$ alone, determine the effective boundary penalties.

\begin{cor}[Coercivity for Cubic Splines]\label{cor:cubic-coercivity}
Let $V_h$ consist of $C^2$ cubic tensor-product B-splines on a uniform Cartesian mesh in $d=2,3$, with $\delta=h$. For any $\tau>0$, $\gamma_3\geq1$, $\beta\geq30$, and sufficiently small $h$ depending only on the boundary geometry,
\begin{equation}\label{eq:cubic-coercivity}
A_{h,j}(v,v)\geq\tfrac12\tn v\tn_{h,j}^2\qquad v\in V_h
\end{equation}
where $\tn v\tn_{h,j}^2=\tn v\tn_h^2+j_h(v,v)$.
\end{cor}
\begin{proof}
Both boundary traces have the regularity required by the Hermite lifting, and face stabilization controls the additional term in cellwise Green's formula. The coercivity argument of Lemma~\ref{lem:continuity-coercivity} then applies with $\|v\|_{a_h}^2$ replaced by $\|v\|_{a_h}^2+j_h(v,v)$; see Appendix~\ref{section:cubic-coercivity}.
\end{proof}

\subsection{Interpolation Estimates}
\label{section:interpolation}

Under Assumption~\ref{ass:analysis}, fix $p\geq4$ and $4\leq s\leq p+1$. On the smooth domain there is a bounded extension operator $E:H^s(\Omega)\to H^s(\IR^d)$ satisfying
\begin{equation}\label{eq:extension-continuity}
\|Ev\|_{H^s(\IR^d)}\lesssim\|v\|_{H^s(\Omega)}
\end{equation}
see \cite{Fo95}. Let $\Pi_h$ be a locally stable spline quasi-interpolation operator on the background grid and define
\begin{equation}\label{eq:interpolant-def}
\pi_h v=\Pi_h(Ev)
\end{equation}
The local approximation estimates are
\begin{equation}\label{eq:interpol-basic}
\|D^m(Ev-\pi_h v)\|_T
\lesssim h^{s-m}\|Ev\|_{H^s(\mcN_h(T))}
\qquad m=0,\ldots,4
\end{equation}
where $D^m$ denotes the collection of Cartesian derivatives of order $m$, with $D^0$ the identity, and $\mcN_h(T)$ is the union of supports of the spline basis functions that are nonzero on $T$. For fixed $p$ these patches have bounded overlap. See \cite{IGA-h-refined} for spline approximation estimates. Derivatives on collections of elements are taken elementwise, and integrals over such collections mean integrals over their union.

\begin{lem}[Interpolation]\label{lem:interpolation}
Under Assumption~\ref{ass:analysis}, for $v\in H^s(\Omega)$, $4\leq s\leq p+1$, we have
\begin{equation}\label{eq:interpol-energy}
\tn v-\pi_h v\tn_h\lesssim h^{s-2}\|v\|_{H^s(\Omega)}
\end{equation}
The constant is independent of how the boundary cuts the mesh.
\end{lem}
\begin{proof}
Set $e=v-\pi_h v$ on $\Omega$ and $\widetilde e=Ev-\pi_h v$ on the background mesh. The local estimates, bounded overlap, and extension stability give
\begin{equation}\label{eq:interpolation-volume}
\|\Delta e\|_\Omega^2
+\tau\delta^4\|\Delta^2 e\|_{D_\delta}^2
\lesssim h^{2(s-2)}\|v\|_{H^s(\Omega)}^2
\end{equation}
For the boundary terms, use the cut-independent trace inequality on each boundary element,
\begin{equation}\label{eq:cut-trace-interpolation}
\|q\|_{T\cap\partial\Omega}^2
\lesssim h^{-1}\|q\|_T^2+h\|\nabla q\|_T^2
\qquad q\in H^1(T)
\end{equation}
valid for a fixed smooth boundary and sufficiently small $h$; see \cite{HaHaLa03}. Apply this inequality to derivatives of $\widetilde e$ of orders zero through three. The coefficients in the surface differential operators and their required derivatives are bounded by the geometry. This yields
\begin{align}\label{eq:interpolation-boundary}
&h^{-3}\|e\|_\Gamma^2
+h^{-1}\bigl(\|\nabla_\Gamma e\|_\Gamma^2+\|\partial_n e\|_\Gamma^2\bigr)
\nonumber\\&\quad
+h\bigl(\|\Delta_\Gamma e\|_\Gamma^2+\|\nabla_\Gamma\partial_n e\|_\Gamma^2\bigr)
+h^3\|\Delta_\Gamma\partial_n e\|_\Gamma^2
\nonumber\\&\qquad\lesssim
\sum_{m=0}^4h^{2m-4}\|D^m\widetilde e\|_{\mcT_h(\Gamma)}^2
\lesssim h^{2(s-2)}\|v\|_{H^s(\Omega)}^2
\end{align}
In particular, the highest-order term requires derivatives of order four:
\begin{equation}\label{eq:other-interp-e}
h^3\|\Delta_\Gamma\partial_n e\|_\Gamma^2
\lesssim\sum_{m=1}^4h^{2m-4}\|D^m\widetilde e\|_{\mcT_h(\Gamma)}^2
\end{equation}
The terms with lower derivatives account for derivatives of the normal and the tangential projection. Since $\delta\sim h$ and $\tau$ is fixed, \eqref{eq:interpolation-boundary} bounds $\|e\|_{b_h}^2$. Combining it with \eqref{eq:interpolation-volume} proves the lemma.
\end{proof}

\subsection{Error Estimates}

\begin{thm}\label{thm:error-estimates}
Under Assumption~\ref{ass:analysis}, let $p\geq4$. If $u\in H^s(\Omega)$, $4\leq s\leq p+1$, then
\begin{equation}\label{eq:errorest-energy}
\tn u-u_h\tn_h\lesssim h^{s-2}\|u\|_{H^s(\Omega)}
\end{equation}
Moreover,
\begin{equation}\label{eq:errorest-L2}
\|u-u_h\|_\Omega\lesssim h^s\|u\|_{H^s(\Omega)}
\end{equation}
The constants are independent of the cut configuration. In particular, $u\in H^{p+1}(\Omega)$ gives orders $p-1$ and $p+1$, respectively.
\end{thm}
\begin{proof}
Put $e=u-u_h$ and use the continuity constant $C_A$. Since $e\in H^4(\Omega)$, coercivity on the full space, Galerkin orthogonality, and continuity imply
\begin{align}
\tfrac12\tn e\tn_h^2
&\leq A_h(e,e)
=A_h(e,u-\pi_hu)
\\&\leq C_A\tn e\tn_h\tn u-\pi_hu\tn_h
\end{align}
Combining this bound with Lemma~\ref{lem:interpolation} yields \eqref{eq:errorest-energy}.

For the $L^2$ estimate, let $z\in H^2_0(\Omega)$ solve the dual problem
\begin{equation}\label{eq:dual-problem}
(\Delta z,\Delta v)_\Omega=(e,v)_\Omega
\qquad\forall v\in H^2_0(\Omega)
\end{equation}
The regularity estimate in Assumption~\ref{ass:analysis} gives $\|z\|_{H^4(\Omega)}\lesssim\|e\|_\Omega$.
Thus $\Delta^2z=e$ in $\Omega$ and $z=\partial_n z=0$ on $\partial\Omega$. 
Green's formula and the vanishing boundary traces of $z$ give
\begin{align}\label{eq:dual-consistency}
A_h(e,z)
&=\|e\|_\Omega^2
+\tau\delta^4(\Delta^2e,e)_{D_\delta}
\end{align}
The additional term is the contribution of the bulk least-squares stabilization and must be retained. By Galerkin orthogonality,
\begin{align}
\|e\|_\Omega^2
&=A_h(e,z-\pi_hz)
-\tau\delta^4(\Delta^2e,e)_{D_\delta}
\\&\leq C_A\tn e\tn_h\tn z-\pi_hz\tn_h
+\sqrt{\tau}\,\delta^2\|e\|_{a_h}\|e\|_\Omega
\\&\lesssim h^2\tn e\tn_h\|e\|_\Omega
\end{align}
The last step uses Lemma~\ref{lem:interpolation} with $s=4$, \eqref{eq:dual-regularity}, $\delta\sim h$, and fixed $\tau$. Division by $\|e\|_\Omega$, unless $e=0$, followed by \eqref{eq:errorest-energy} proves \eqref{eq:errorest-L2}.
\end{proof}

\subsection{Error Estimates for Lower-Degree Splines}\label{section:face-analysis}
The extended method \eqref{eq:H2-method} covers $C^1$ quadratic and $C^2$ cubic splines, with coercivity for the latter established in Corollary~\ref{cor:cubic-coercivity}. For $C^1$ quadratics, the normal trace has only $H^1(\Gamma)$ regularity, so Appendix~\ref{section:H2-coercivity} regularizes its lifting. To state an explicit sufficient condition for coercivity, put
\begin{equation}\label{eq:H2-geometric-condition}
\theta=\delta\|\kappa\|_{L^\infty(\Gamma)},
\qquad \sigma=\delta^2\|\nabla_\Gamma\kappa\|_{L^\infty(\Gamma)}
\end{equation}
where $\kappa=\nabla_\Gamma n$ is the shape operator and the tensor norms are Frobenius norms.

\begin{lem}[Coercivity for Quadratic Splines]\label{lem:quadratic-coercivity}
Let $V_h$ consist of $C^1$ quadratic tensor-product B-splines on a uniform Cartesian mesh in $d=2,3$, with $\delta=h$ below the two-sided tubular radius. If $\theta,\sigma\leq1/10$, then for any $\tau>0$, $\gamma_2\geq1$, and $\beta\geq30$,
\begin{equation}\label{eq:H2-spline-coercivity}
A_{h,j}(v,v)\geq\tfrac15\tn v\tn_{h,j}^2\qquad v\in V_h
\end{equation}
The bound is independent of the cut configuration and $\tau$.
\end{lem}
\begin{proof}
Apply Young's inequality to the bound \eqref{eq:H2-Q-bound} from Appendix~\ref{section:quadratic-coercivity}.
\end{proof}
The bounds $\theta,\sigma\leq1/10$ are sufficient, not necessary; within the tubular radius, repeating the estimates of Appendix~\ref{section:physical-norms} with the actual values $\theta,\sigma<1$ gives a geometry-dependent, possibly larger sufficient $\beta$.

\begin{thm}[Error Estimates for Lower-Degree Splines]\label{thm:face-error}
Let $V_h$ be the full tensor-product B-spline space of degree $p=2,3$ and continuity $C^{p-1}$. Assume the hypotheses of Lemma~\ref{lem:quadratic-coercivity} for $p=2$ or Corollary~\ref{cor:cubic-coercivity} for $p=3$, with fixed parameters, $\delta=h\leq1$, and the trace inequality \eqref{eq:cut-trace-interpolation}. If $u\in H^4(\Omega)$, the solution of \eqref{eq:H2-method} satisfies
\begin{equation}\label{eq:H2-error-energy}
\tn u-u_h\tn_{h,j}\lesssim h^{p-1}\|u\|_{H^4(\Omega)}
\end{equation}
With the regularity estimate \eqref{eq:dual-regularity},
\begin{equation}\label{eq:H2-error-L2}
\|u-u_h\|_\Omega\lesssim h^{2p-2}\|u\|_{H^4(\Omega)}
\end{equation}
The constants are independent of the cut configuration.
\end{thm}
\begin{proof}
The energy estimate follows from C\'ea's lemma and spline interpolation; the $L^2$ estimate follows by duality. See Appendix~\ref{section:face-error-proof} for details.
\end{proof}

\section{Numerical Results}\label{section:numerics}
\subsection{Setup and Diagnostics}
We use $C^3$ quartic tensor-product B-splines, with $\beta=30$, $\tau=0.01$, and $\delta=h$ unless stated otherwise. All experiments use the exact ellipse
\begin{equation}\label{eq:ellipse-domain}
\Omega=\{x:x_1^2/a^2+x_2^2/b^2<1\},\qquad a=3/2,\quad b=1
\end{equation}
Throughout this section, $h$ denotes the cell diameter and $\ell$ the Cartesian cell width:
\begin{equation}\label{eq:numerical-mesh-size}
\ell=\frac{2b}{N},\qquad h=\sqrt2\ell=\frac{2\sqrt2 b}{N}
\end{equation}
All stabilization widths, penalty scalings, and removal budgets use $h$; grid translations are measured in units of $\ell$. The grid lines are $x_i=-a+(i+s_x)\ell$, $y_j=-b+(j+s_y)\ell$. The default phase is $(s_x,s_y)=(0.23,0.37)$. Exact boundary parametrization and normal coordinates are used for boundary and strip integration, with intervals split at cell crossings. Assembly and error evaluation use 12 and 18 Gauss points per direction, respectively. All curvature contributions are retained. Geometric moments, boundary differentiation, polynomial reproduction, quadrature refinement, and independent spectral calculations provide verification checks. Their settings and representative numerical values are given in Appendix~\ref{section:numerical-verification}. On these meshes we assess discrete coercivity directly; we do not evaluate the sufficient geometry constants of the $H^4$ analysis in Appendix~\ref{section:cont-consistency}.

\paragraph{Comparison with Standard Nitsche.}
The standard method retains only the volume Laplacian form, symmetric consistency terms, and the value and normal-derivative penalties. In the fixed-method study below, its penalties are calibrated to give a positive definite coarse-grid matrix and then held fixed. In the other comparisons they match those of LS Nitsche: $\beta(16+1/(3\tau))\delta^{-3}$ and $\beta(16+1/(3\tau))/(30\delta)$, equal to $1480h^{-3}$ and $(148/3)h^{-1}$ at the default parameters. This matched comparison is a diagnostic of the stabilization terms, not a comparison with an independently tuned method; coercivity of standard Nitsche is not ensured by these coefficients.

\paragraph{Conditioning and Basis Function Removal.}
We use symmetric diagonal scaling and, following \cite{ElfLarLar18}, remove the smallest-energy basis functions subject to
\begin{equation}\label{eq:basis-removal-budget}
\sum_{i\in I_r}\tn\varphi_i\tn_h^2\leq(0.01h^{p-1})^2
\end{equation}
This selection uses the diagonal of the energy Gram matrix $G_{ij}=a_h(\varphi_i,\varphi_j)+b_h(\varphi_i,\varphi_j)$, with no multiplier $\beta$. In Section~\ref{section:H2-numerics}, $G$ additionally includes $j_h(\varphi_i,\varphi_j)$. All compared methods use the same retained space. Removal is active in the convergence and cut-position studies and disabled in the fixed-method and parameter studies so that their spaces stay fixed as the coefficients change. The reported convergence curves therefore concern the retained spaces with budget \eqref{eq:basis-removal-budget}; the full-space error estimates concern the unmodified spline space. Basis removal is a separate conditioning measure and is not needed for the coercivity proofs.

\paragraph{Stability Diagnostic.}
Let $A$ be the stiffness matrix of the method under consideration, restricted to the retained spline basis, and restrict $G$ to the same basis. We report the minimum generalized eigenvalue
\begin{equation}\label{eq:numerical-eigenvalue}
\lambda_{\min}(A,G)=\min_{\mathbf v\neq0}
\frac{\mathbf v^TA\mathbf v}{\mathbf v^TG\mathbf v}
\end{equation}
All comparisons use the same $G$, including for standard Nitsche. We locate the minimum by bisection of the positive-definiteness threshold of $A-\lambda G$ and report estimated 1-norm condition numbers after symmetric diagonal scaling. A diagnostic is marked unresolved if its scaled condition estimate times machine precision exceeds $10^{-4}$. A positive lower bound on \eqref{eq:numerical-eigenvalue} measures coercivity in the physical energy norm; matrix conditioning also depends on the basis and scaling. The analysis does not provide a cut-independent matrix condition-number bound. Solver settings and representative eigenvalues and condition estimates are recorded in Appendix~\ref{section:numerical-verification}.

\subsection{Convergence}
We prescribe
\begin{align}
u(x_1,x_2)&=0.1\sin(2x_1)+0.1x_1\cos(3x_2)\\
f(x_1,x_2)&=\Delta^2u=1.6\sin(2x_1)+8.1x_1\cos(3x_2)
\end{align}
with the corresponding nonhomogeneous traces. Figure~\ref{fig:ellipse-solution} shows the solution and active mesh at $N=8$.
\begin{figure}[htbp]
\centering
\includegraphics[width=0.53\textwidth]{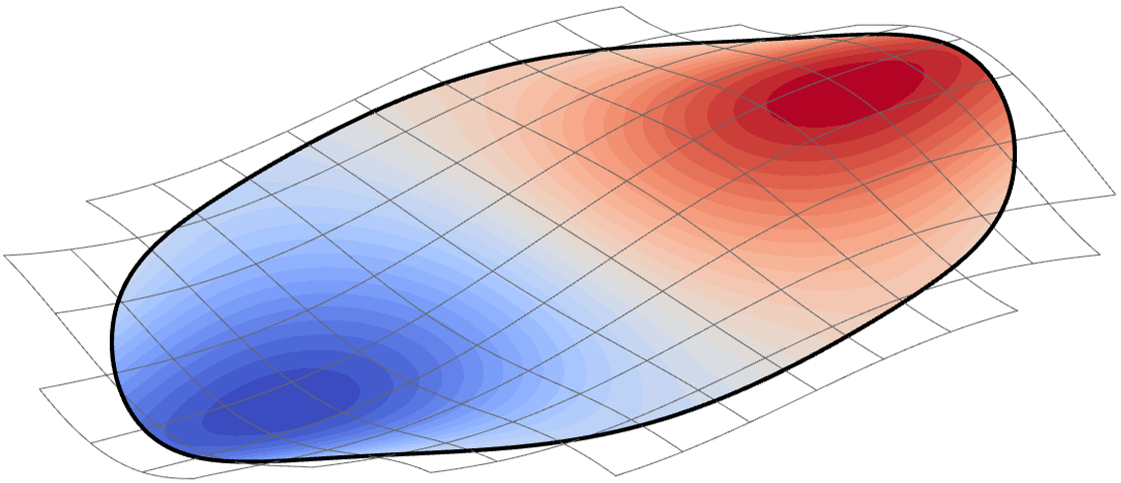}
\caption{Computed solution on the exact ellipse and its active mesh at $N=8$.}\label{fig:ellipse-solution}
\end{figure}

Table~\ref{tab:ellipse-convergence} reports the full energy, Laplacian, and $L^2$ errors. The Laplacian rate approaches the predicted order three. The observed full energy rate is about $3.4$--$3.5$ on the last refinements. We interpret this through the dominant bulk residual term: integration of a smooth error profile over a strip of width $h$ can supply an additional factor $h^{1/2}$ in its $L^2$ norm. The proven energy order remains three. The $L^2$ rates are consistent with order five.

\begin{table}[htbp]
\centering\small
\caption{Convergence with $C^3$ quartic splines and the default parameters. The removed basis functions satisfy $\sum_{i\in I_r}\tn\varphi_i\tn_h^2\leq(0.01h^3)^2$. Rates are between consecutive refinements.}\label{tab:ellipse-convergence}
\setlength{\tabcolsep}{3.5pt}
\begin{tabular}{@{}rrrrrrrr@{}}\toprule
$N$ & DoFs (Removed) & Energy & Rate & Laplacian & Rate & $L^2$ & Rate\\\midrule
8 & 194 (8) & $2.32\cdot10^{-2}$ & -- & $3.75\cdot10^{-3}$ & -- & $2.92\cdot10^{-5}$ & -- \\
16 & 519 (9) & $2.27\cdot10^{-3}$ & 3.35 & $3.81\cdot10^{-4}$ & 3.30 & $5.85\cdot10^{-7}$ & 5.64 \\
32 & 1626 (5) & $2.11\cdot10^{-4}$ & 3.43 & $4.33\cdot10^{-5}$ & 3.14 & $1.13\cdot10^{-8}$ & 5.69 \\
64 & 5641 (13) & $1.93\cdot10^{-5}$ & 3.45 & $5.24\cdot10^{-6}$ & 3.05 & $2.51\cdot10^{-10}$ & 5.49 \\
\bottomrule\end{tabular}
\end{table}

We also repeat refinement for the 25 phases $(s_x,s_y)\in\{0.07,0.27,0.47,0.67,0.87\}^2$. Figure~\ref{fig:worst-cut-convergence} reports the largest errors over these translations. Their variation is small, and the maximum errors have essentially the same rates as Table~\ref{tab:ellipse-convergence}.
\begin{figure}[t]
\centering
\includegraphics[width=0.66\textwidth]{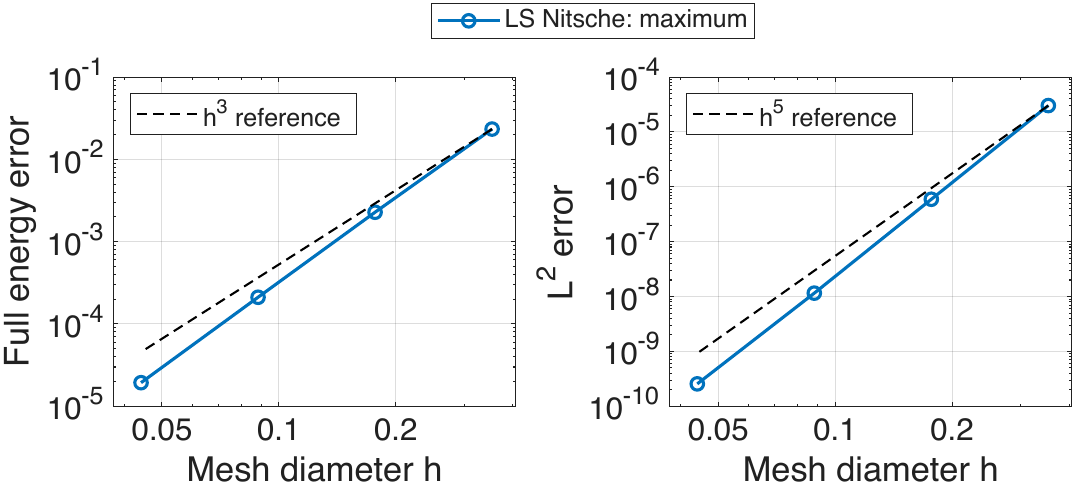}
\caption{Largest errors over 25 translations at each mesh size, with the default parameters and removal budget \eqref{eq:basis-removal-budget}. Dashed lines give the $h^3$ and $h^5$ reference rates.}\label{fig:worst-cut-convergence}
\end{figure}

\subsection{Cut Dependence and Coercivity}\label{section:cut-numerics}
\paragraph{Coercivity with the Method Held Fixed.}\label{section:fixed-method}
Following \cite{ElfLarLar19}, we refine the approximation space while keeping every coefficient and integration region in the bilinear form fixed. Set $\delta=\delta_{\rm fix}=\sqrt2/4$, $\beta=30$, and $\tau=0.01$. The phases $(N/8)\,(0.27,0.47)\bmod1$, for $N=8,16,32,64$, give nested grids. Removal is disabled. We multiply the two standard Nitsche penalties by the first factor in $1,10,100,1000$ that makes its coarse-grid matrix positive definite; this gives factor $1000$, held fixed thereafter.

Figure~\ref{fig:fixed-method} shows that LS Nitsche retains a minimum generalized eigenvalue of about $0.90$, whereas the calibrated standard method becomes indefinite at the first refinement. Thus coarse-grid calibration alone does not ensure stability on richer spaces for the same form. The fixed width satisfies $\delta_{\rm fix}\|\kappa\|_\infty=3\sqrt2/8>1/2$, beyond the auxiliary geometric bound in Appendix~\ref{section:cont-consistency}; the experiment demonstrates discrete stability on the tested spaces, not full-space coercivity at this width.
\begin{figure}[htbp]
\centering
\includegraphics[width=0.66\textwidth]{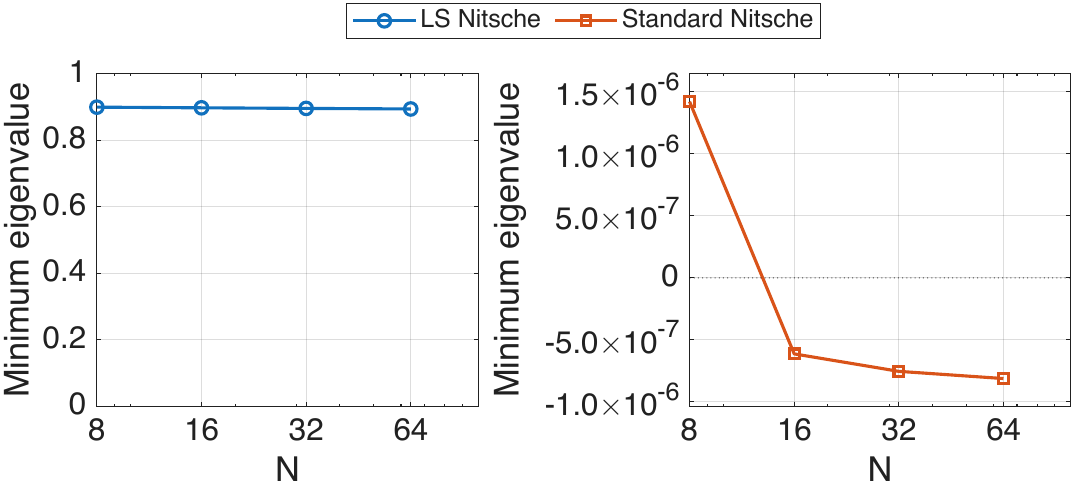}
\caption{Minimum generalized eigenvalues under nested refinement with fixed $\delta=\sqrt2/4$, $\beta=30$, and $\tau=0.01$, without removal. The standard penalties are calibrated on the coarse grid and multiplied by $1000$ at every level. Both methods use the same fixed energy metric; the vertical scales differ.}\label{fig:fixed-method}
\end{figure}

\paragraph{Changing the Cut Position.}
At $N=12$ and $\delta=h$, we use the 25 translations above and two additional families of seven cuts. In the first, a grid vertex approaches $(a,b)/\sqrt2$ from inside the ellipse; in the second, a vertical grid line approaches $x=a$. Their offsets are $10^{-j}\ell$, $j=1,\ldots,7$, producing physical cell-area fractions down to $2.1\cdot10^{-14}$. All 39 LS Nitsche systems are positive definite, with minimum generalized eigenvalues about $0.90$. The matched standard systems are indefinite, with values about $-0.01$; see Figure~\ref{fig:ellipse-experiment}. The condition estimates in Figure~\ref{fig:ellipse-experiment} (right) still vary with the cut, illustrating that coercivity in the physical norm does not by itself ensure uniform matrix conditioning.
\begin{figure}[htbp]
\centering
\includegraphics[width=0.66\textwidth]{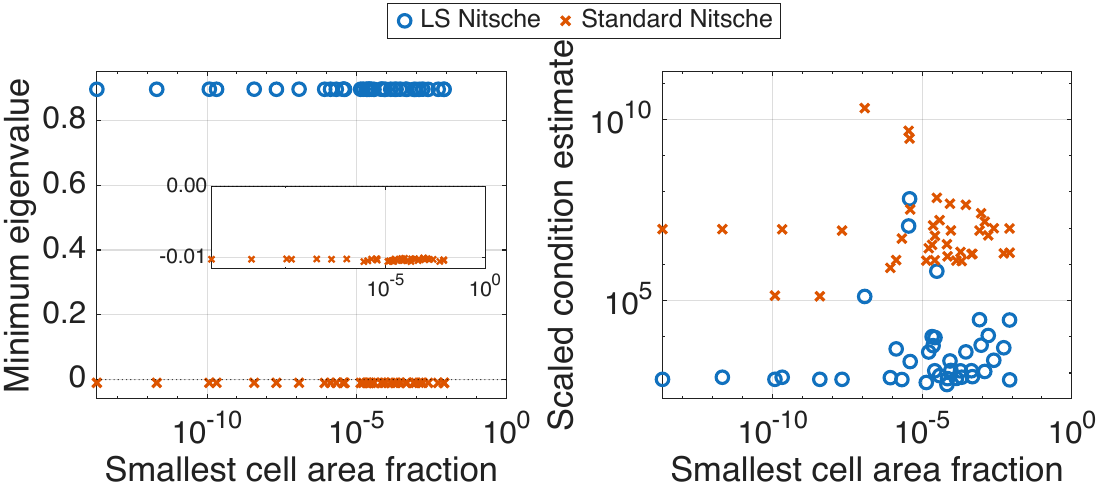}
\caption{The 39 cut-position tests with matched boundary penalties and the same retained basis. Left: minimum generalized eigenvalues; the inset magnifies the negative standard values. Right: scaled condition estimates.}\label{fig:ellipse-experiment}
\end{figure}

\subsection{Parameter Sensitivity}\label{section:parameter-study}
For $N=8,16,32$ at the default phase and $\delta=h$, we vary $\tau=10^{-j}$, $j=0,\ldots,6$, at $\beta=30$, and separately $\beta\in\{3,30,60,300,1200,6000,30000\}$ at $\tau=0.01$. Removal is disabled. To compare errors in a fixed norm, all energy errors in Figure~\ref{fig:parameter-study} use the reference value $\tau=0.01$, independently of the parameter in the method; the norm contains no multiplier $\beta$.

Decreasing $\tau$ initially improves the $L^2$ error, but very small values are detrimental. At $N=32$, the errors for $\tau=10^{-2},10^{-3},10^{-6}$ are $1.1\cdot10^{-8}$, $5.6\cdot10^{-9}$, and $1.0\cdot10^{-7}$, respectively, while the corresponding condition estimates are $1.8\cdot10^4$, $2.6\cdot10^4$, and $3.0\cdot10^7$. Decreasing $\tau$ reduces the bulk residual weight and increases the two standard trace penalties through $16+1/(3\tau)$. Among the tested values, $\tau=10^{-3}$ gives the smallest $L^2$ error at every mesh size; the smallest reference energy errors occur at $\tau=10^{-2}$ or $10^{-1}$.

At $\beta=3$, all three LS Nitsche matrices are indefinite, with minimum generalized eigenvalues $-0.039$, $-0.048$, and $-0.054$ for $N=8,16,32$, respectively. The $L^2$ errors are about $250$, $21$, and $8.5$ times those at $\beta=30$. These cases illustrate loss of coercivity and accuracy when the penalty is too small; their linear solves and spectral diagnostics are resolved. All other tested LS Nitsche matrices are positive definite. The proof of Lemma~\ref{lem:continuity-coercivity} gives the sufficient condition $\beta\geq2c_\delta^2+1/2$ for coercivity constant $1/2$; its right-hand side tends to $17$ as $\delta\to0$.

Increasing $\beta$ also eventually reduces accuracy: at $N=32$, changing $\beta$ from $30$ to $30000$ increases the $L^2$ error from $1.1\cdot10^{-8}$ to $6.9\cdot10^{-7}$. This trend is consistent with locking as nonhomogeneous traces on the curved boundary are imposed increasingly strongly. For this problem the tests support $\tau\in[10^{-3},10^{-2}]$ and $\beta=30$ as practical choices balancing accuracy and conditioning; they do not identify universal optimal parameters.
\begin{figure}[htbp]
\centering
\includegraphics[width=\textwidth]{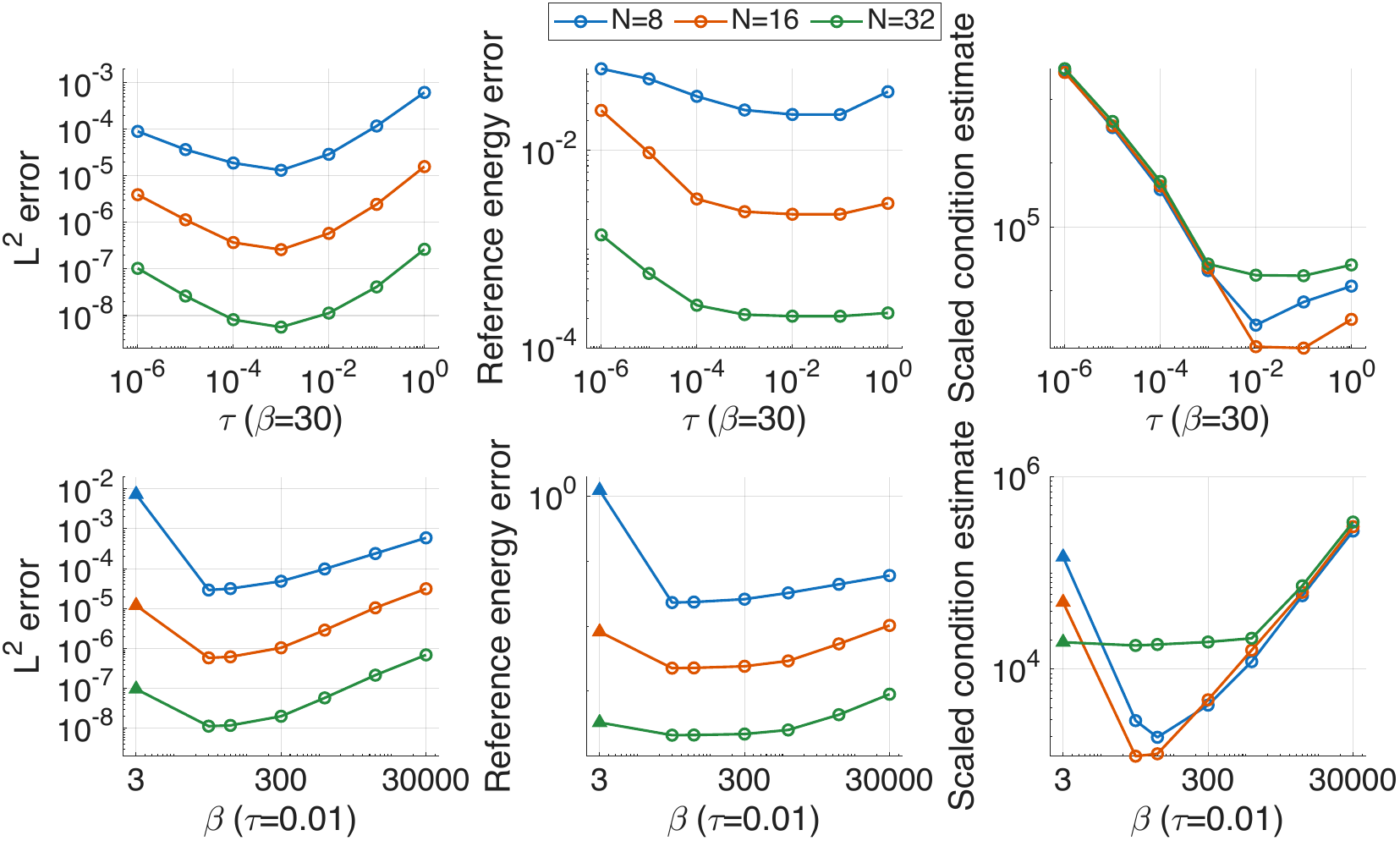}
\caption{Parameter sensitivity without removal. Top: varying $\tau$ at $\beta=30$. Bottom: varying $\beta$ at $\tau=0.01$; filled triangles mark the indefinite systems at $\beta=3$. Energy errors use the common reference value $\tau=0.01$.}\label{fig:parameter-study}
\end{figure}

\FloatBarrier
\subsection{Lower-Degree Splines}\label{section:H2-numerics}
The $C^2$ cubic extension is established analytically in Corollary~\ref{cor:cubic-coercivity} and Theorem~\ref{thm:face-error}; the tests in this subsection concern $C^1$ quadratics. We test \eqref{eq:H2-method} with $C^1$ quadratic splines, $\gamma_2=1$, and the default geometry and parameters. Only the second normal-derivative jump is penalized; the elementwise bi-Laplacian need not vanish. Removal uses \eqref{eq:basis-removal-budget} with $p=2$ and the augmented norm $\tn \cdot\tn_{h,j}$. All three methods in Figure~\ref{fig:H2-stability} use this same retained space and metric. Face stabilization raises the minimum generalized eigenvalues across the 39 cuts from $0.029$--$0.16$ to about $0.98$; all resolved matched standard systems are indefinite. Under refinement, the stabilized energy and $L^2$ rates reach about $1.3$ and $2.5$, consistent with the guaranteed orders one and two. Adding face stabilization improves the coercivity margin, but does not uniformly reduce the $L^2$ error. The sufficient conditions $\theta,\sigma\leq1/10$ hold for this ellipse when $N\geq43$: $\|\kappa\|_\infty=3/2$, $\|\nabla_\Gamma\kappa\|_\infty\leq45/16$, and $h=2\sqrt2/N$. Thus the finest tested mesh, $N=64$, lies in this certified coercivity regime; on the coarser meshes we assess coercivity numerically.
\begin{figure}[htbp]
\centering
\includegraphics[width=0.66\textwidth]{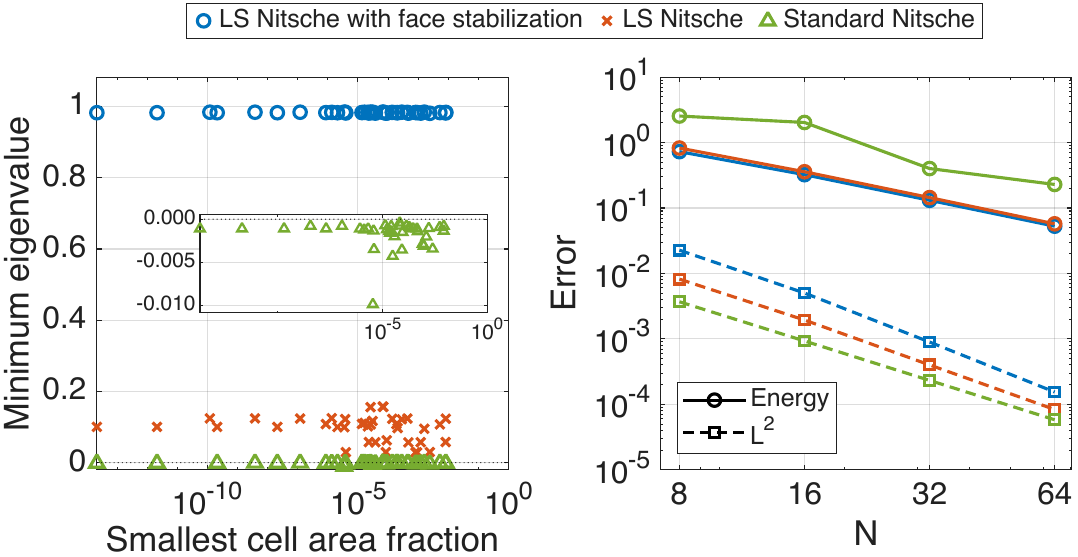}
\caption{$C^1$ quadratic comparison with matched boundary penalties. Left: minimum generalized eigenvalues in the augmented norm; the inset magnifies the standard values. One standard cut is omitted because its scaled condition estimate times machine precision exceeds $10^{-4}$. Right: energy errors (solid circles) and $L^2$ errors (dashed squares), with method colors shared by both panels.}\label{fig:H2-stability}
\end{figure}

\FloatBarrier
\section{Conclusions}\label{section:conclusions}
The explicit lifting of the two Dirichlet traces provides a symmetric LS Nitsche method for the biharmonic problem with coercivity on $H^4(\Omega)$, without inverse inequalities. For sufficiently thin boundary strips and a chosen bulk stabilization parameter, fixed, cut-independent penalties suffice. The lifting determines the relative weights of the value, normal-derivative, and tangential penalties, and gives energy and $L^2$ error orders $p-1$ and $p+1$ for splines of degree $p\geq4$. Face stabilization extends the analysis to maximum-continuity splines of degrees $p=2,3$. Two-dimensional exact-ellipse experiments with quartic and quadratic splines support the stability and approximation results and show how insufficient penalties cause loss of coercivity, while excessive penalties impair accuracy. The three-dimensional and cubic-spline results remain analytical; cut-independent matrix conditioning is a separate issue from the coercivity established here.

The stabilization could also be extended from boundary conditions to interface conditions. For higher-order problems this is of particular interest in Kirchhoff--Love shell configurations, where displacement and rotation continuity must be imposed across patches; see \cite{JoLaLa17} for Nitsche coupling on multipatch surfaces in the second-order case. Further questions concern liftings near corners and edges on piecewise smooth domains, and stability and accuracy when the boundary and its curvature are approximated; see \cite{LarLarMah26} for error estimates for Nitsche's method on approximate domains in the second-order case.

\appendix
\small

\section{The Boundary Lifting}\label{section:cont-consistency}
We prove Lemma~\ref{lem:const-bound} and Corollary~\ref{cor:cubic-coercivity}. Generic constants depend only on the fixed boundary geometry and tubular radius.

\subsection{Normal Coordinates}
\label{section:tech-tools}

Choose $\delta_0>0$ below the tubular radius of the smooth boundary so that
\begin{equation}\label{eq:normal-coordinates}
\Phi(y,r)=y-rn(y)
\qquad (y,r)\in\Gamma\times(0,\delta_0)
\end{equation}
is a diffeomorphism onto $D_{\delta_0}$. Thus $r$ is the distance measured inward from $\Gamma$, and $\partial_n=-\partial_r$ on $\Gamma$. Let $\kappa=\nabla_\Gamma n$ denote the shape operator, restricted to the tangent bundle of $\Gamma$, and set
\begin{equation}\label{eq:tube-metric}
M_r=(I_\Gamma-r\kappa)^{-1},
\qquad J_r=\det(I_\Gamma-r\kappa)
\end{equation}
where $I_\Gamma$ is the tangential identity. We decrease $\delta_0$, if necessary, so that $\delta_0\|\kappa\|_{L^\infty(\Gamma)}\leq 1/2$.
The volume element is $dx=J_r\,d\Gamma\,dr$. Surface Hessians and derivatives of tangential tensors below are understood covariantly. These normal coordinates and the associated surface calculus are described in \cite[Section~2]{DzEl13}.

Let $\Gamma_r=\Phi(\Gamma,r)$ and denote its Laplace--Beltrami operator, pulled back to $\Gamma$, by $\Delta_{\Gamma_r}$, so that $\Delta_{\Gamma_0}=\Delta_\Gamma$. For $Z=z\circ\Phi$, the Laplacian in these coordinates is
\begin{align}\label{eq:tube-laplacian}
(\Delta z)\circ\Phi
&=\partial_{rr}Z+(\partial_r\log J_r)\partial_r Z
+J_r^{-1}\nabla_\Gamma\cdot(J_rM_r^2\nabla_\Gamma Z)\\
&=\partial_{rr}Z+\Delta_\Gamma Z
+\underbrace{(\partial_r\log J_r)\partial_r Z
+(\Delta_{\Gamma_r}-\Delta_\Gamma)Z}_{:=R_rZ}
\end{align}
The identity follows from the divergence formula in the metric induced by $\Phi$. In particular, it retains the normal first-derivative term. Smoothness of $\kappa$ gives constants depending only on the fixed geometry such that
\begin{align}\label{eq:tube-coefficient-bounds}
0<J_r&\leq 1+C_0\delta,
\qquad 0\leq r\leq\delta\leq\delta_0
\\
|R_r Z|&\leq C_\Gamma\bigl(|\partial_r Z|
+r|\nabla_\Gamma^2 Z|+r|\nabla_\Gamma Z|\bigr)
\label{eq:tube-remainder}
\end{align}
Here $\nabla_\Gamma^2$ is the covariant surface Hessian. Indeed, $\partial_r\log J_r=-\operatorname{tr}(M_r\kappa)$ is uniformly bounded, while $M_r^2-I_\Gamma$ and $J_r^{-1}\nabla_\Gamma\cdot(J_rM_r^2)$ are bounded by $C_\Gamma r$.

We will also use
\begin{equation}\label{eq:surface-hessian-bound}
\|\nabla_\Gamma^2 g\|_\Gamma
\leq \|\Delta_\Gamma g\|_\Gamma+C_\Gamma\|\nabla_\Gamma g\|_\Gamma
\qquad g\in H^2(\Gamma)
\end{equation}
To see this, integration by parts and commutation of covariant derivatives on the closed surface give
\begin{equation}
\|\nabla_\Gamma^2 g\|_\Gamma^2
=\|\Delta_\Gamma g\|_\Gamma^2
-\int_\Gamma\operatorname{Ric}(\nabla_\Gamma g,\nabla_\Gamma g)
\end{equation}
where $\operatorname{Ric}$ is the Ricci tensor of $\Gamma$. Its boundedness gives \eqref{eq:surface-hessian-bound}, first for smooth functions and then by density.

\subsection{Hermite Lifting and Its Bounds}\label{section:hermite-lifting}
For the traces $g_0=w|_\Gamma$ and $g_1=\partial_nw|_\Gamma$ of $w\in V$, define
\begin{equation}\label{eq:hermite-polynomials}
\phi_0(t)=1-3t^2+2t^3,\qquad \phi_1(t)=-t(1-t)^2
\end{equation}
and
\begin{equation}\label{eq:hermite-lifting}
Z(y,r)=z(\Phi(y,r))=\phi_0(r/\delta)g_0(y)+\delta\phi_1(r/\delta)g_1(y)
\end{equation}
Extend $z$ by zero outside $D_\delta$. The endpoint values give
\begin{equation}\label{eq:lifting-traces}
z=w,\quad\partial_nz=\partial_nw\text{ on }\Gamma,
\qquad Z=\partial_rZ=0\text{ at }r=\delta
\end{equation}
so this extension belongs to $H^2(\Omega)$. The quantities $B_0,B_1$ below are those in \eqref{eq:boundary-quadratic-norm}, evaluated on this pair of traces.

\begin{lem}[Product-Coordinate Lifting Bounds]\label{lem:product-lifting}
With $\|\cdot\|_\times$ the norm for the product measure on $\Gamma\times(0,\delta)$,
\begin{equation}\label{eq:product-lifting-bounds}
\|Z_{rr}+\Delta_\Gamma Z\|_\times^2\leq\frac{33}{4}B_0,
\qquad \delta^{-4}\|Z\|_\times^2\leq\frac{69}{35}B_1
\end{equation}
\end{lem}
\begin{proof}
Direct integration gives the moments
\begin{equation}\label{eq:hermite-moments}
\begin{array}{@{}cccc@{}}\toprule
 &\int_0^1\phi_i^2&\int_0^1(\phi_i')^2&\int_0^1(\phi_i'')^2\\\midrule
i=0&13/35&6/5&12\\
i=1&1/105&2/15&4\\\bottomrule
\end{array}
\end{equation}
Write $Z=Z_0+Z_1$ for the value and normal parts of \eqref{eq:hermite-lifting}. Integration by parts on $\Gamma$ and in $r$ gives
\begin{equation}\label{eq:hermite-Gram}
\|Z_{i,rr}+\Delta_\Gamma Z_i\|_\times^2
=\|Z_{i,rr}\|_\times^2+2\|\nabla_\Gamma Z_{i,r}\|_\times^2+\|\Delta_\Gamma Z_i\|_\times^2
\qquad i=0,1
\end{equation}
The endpoint terms vanish since $\phi_0'$ and $\phi_1$ vanish at both endpoints. By \eqref{eq:hermite-moments}, comparison with the corresponding value and normal terms of $B_0$ gives factors $3/4$ and $15/2$, respectively. Weighted Cauchy--Schwarz between the two parts gives $33/4=3/4+15/2$. The same argument for the mass terms gives
\begin{equation}
\delta^{-4}\|Z\|_\times^2\leq
\left(\frac{39}{35}+\frac{90}{105}\right)B_1=\frac{69}{35}B_1
\end{equation}
\end{proof}

\paragraph{Geometric Remainder.}
The moments \eqref{eq:hermite-moments}, $r\leq\delta$, and the surface Hessian estimate \eqref{eq:surface-hessian-bound} imply
\begin{equation}\label{eq:lifting-remainder-bound}
\|R_rZ\|_\times
\leq C_\Gamma\bigl(\|Z_r\|_\times+
\|r\nabla_\Gamma^2Z\|_\times+\|r\nabla_\Gamma Z\|_\times\bigr)
\leq C_1\delta B_0^{1/2}
\end{equation}
For example, $\|Z_r\|_\times\lesssim\delta^{-1/2}\|g_0\|_\Gamma+
\delta^{1/2}\|g_1\|_\Gamma\lesssim\delta B_0^{1/2}$; the surface terms follow from their $\delta^{3/2}$ and $\delta^{5/2}$ factors and the corresponding terms of $B_0$.
The Jacobian bound in \eqref{eq:tube-coefficient-bounds}, Lemma~\ref{lem:product-lifting}, and $69/35<33/4$ therefore yield
\begin{equation}\label{eq:lifting-bound}
\|\Delta z\|_{D_\delta}^2+\tau^{-1}\delta^{-4}\|z\|_{D_\delta}^2
\leq c_\delta^2(B_0+\tau^{-1}B_1)
\end{equation}
where an admissible constant is
\begin{equation}\label{eq:constant}
c_\delta=(1+C_0\delta)^{1/2}(\sqrt{33}/2+C_1\delta),
\qquad c_\delta^2\leq\tfrac{33}{4}+C\delta
\end{equation}
The latter bound holds on the fixed interval $0<\delta\leq\delta_0$, with the same geometry constant $C$ as in \eqref{eq:trace}.

\begin{proof}[Proof of Lemma~\ref{lem:const-bound}]
Apply the identity \eqref{eq:lifting-identity-main} and weighted Cauchy--Schwarz \eqref{eq:lifting-CS-main} to the lifting above, then use \eqref{eq:lifting-bound}. Green's formula is valid for this $H^2$ lifting by density; no fourth derivatives of $z$ or discrete inverse inequalities are needed.
\end{proof}

\subsection{Coercivity for Cubic Splines}\label{section:cubic-coercivity}
The grid consists of cubes of side length $\ell$, with $\delta=h=\sqrt d\ell$, $d=2,3$.
For compactly supported $f\in H^1(\IR)$, integration of $(|f|^2)'$ on each grid interval gives, uniformly in the translation $b$,
\begin{equation}\label{eq:lattice-trace}
\sum_{i\in\mathbb Z}|f(b+i\ell)|^2\leq\ell^{-1}\|f\|_{\IR}^2+2\|f\|_{\IR}\|f'\|_{\IR}
\end{equation}

\begin{proof}[Proof of Corollary~\ref{cor:cubic-coercivity}]
Both boundary traces belong to $H^2(\Gamma)$, so the Hermite lifting above gives the volume bound \eqref{eq:lifting-bound}. Extend its product representation to $\bar Z$ on the two-sided tube by even reflection of the value part and odd reflection of the normal part. For the two-sided product norm, the moments \eqref{eq:hermite-moments} and parity give $\delta^{-4}\|\bar Z\|_\times^2\leq MB_0$ and $\delta^{-2}\|\nabla_\times\bar Z\|_\times^2\leq GB_0$, with $M=13/280$, $G=1/2$, and $\nabla_\times=(\nabla_\Gamma,\partial_r)$. The metric changes these by at most a factor $1+C\delta$. The lattice trace estimate \eqref{eq:lattice-trace}, applied in each coordinate direction, therefore gives
\begin{equation}\label{eq:cubic-face-bound}
\sum_{F\in\mcF_{h,\delta}}h^{-3}\|z\|_{F\cap\Omega}^2
\leq C_3(\delta)B_0,
\qquad C_3(\delta)\leq\sqrt3(3M+2\sqrt{MG})(1+C\delta)<\tfrac45(1+C\delta)
\end{equation}
For $\gamma_3\geq1$ and sufficiently small $h$, the two-sided tube lies inside the tubular neighborhood and $c_\delta^2+C_3(\delta)/\gamma_3\leq12$. Since $\ldb\Delta v\rdb=0$ and $\ldb\partial_{n_F}\Delta v\rdb=\ldb\partial_{n_F}^3v\rdb$, cellwise Green's formula and weighted Cauchy--Schwarz yield
\begin{equation}\label{eq:cubic-consistency}
|Q(v,w)|\leq\sqrt{12}\bigl(\|v\|_{a_h}^2+j_h(v,v)\bigr)^{1/2}\|w\|_{b_h}
\end{equation}
Young's inequality as in Lemma~\ref{lem:continuity-coercivity} proves the claim.
\end{proof}

\section{Estimates for Quadratic and Cubic Splines}\label{section:H2-coercivity}
Coercivity for cubic splines is established in Corollary~\ref{cor:cubic-coercivity}. Here we prove the quadratic coercivity estimate of Lemma~\ref{lem:quadratic-coercivity} and the error estimates of Theorem~\ref{thm:face-error}. As in the analysis for $H^4$-conforming approximation spaces, the stability argument does not invoke discrete inverse inequalities. The grid consists of cubes of side length $\ell$, with $\delta=h=\sqrt d\ell$, $d=2,3$. We use the two-sided normal coordinates $\Gamma\times(-\delta,\delta)$ and the geometry parameters \eqref{eq:H2-geometric-condition}.

\subsection{Lifting the Normal Trace}
For $C^1$ quadratics, $g_0=w|_\Gamma\in H^2(\Gamma)$ and $g_1=\partial_nw|_\Gamma\in H^1(\Gamma)$. Keep the value lifting of Appendix~\ref{section:cont-consistency}, and regularize only the normal part. In an orthonormal eigenbasis $e_m$ of $-\Delta_\Gamma$ with eigenvalues $\lambda_m$, write $g_0=\sum_ma_me_m$, $g_1=\sum_mb_me_m$, and set
\begin{equation}\label{eq:H2-mode-width}
q_m=\delta^2\lambda_m,\qquad \mu_m=\sqrt{1+q_m},\qquad d_m=\delta/\mu_m
\end{equation}
Extend the profiles \eqref{eq:hermite-polynomials} by zero for arguments greater than one and define
\begin{equation}\label{eq:H2-regularized-lifting}
\bar Z(y,r)=\phi_0(|r|/\delta)g_0(y)
+\sum_m d_m\operatorname{sgn}(r)\phi_1(|r|/d_m)b_me_m(y)
\end{equation}
Let $\bar z=\bar Z\circ\Phi^{-1}$, extended by zero outside the tube, $z=\bar z|_\Omega$, and $Z=\bar Z|_{r>0}$. Put
\begin{equation}\label{eq:H2-mode-boundary-weight}
\mathcal P_0=\sum_m\delta^{-3}(16+4q_m+q_m^2)|a_m|^2,
\qquad \mathcal P_1=\sum_m\delta^{-1}\frac{16+4q_m}{30}|b_m|^2
\end{equation}
Then $\mathcal B=\mathcal P_0+\mathcal P_1\leq B_0(w)$. The broken $\Delta_\Gamma g_1$ penalty is unused in this bound and is retained in the method.

Write $\|\cdot\|_\times$ for the norm on $\Gamma\times(-\delta,\delta)$, $\nabla_\times=(\nabla_\Gamma,\partial_r)$, and $T_2^2=2\|\nabla_\Gamma\bar Z_r\|_\times^2+\|\Delta_\Gamma\bar Z\|_\times^2$. Parity eliminates cross terms between the two parts. The moments \eqref{eq:hermite-moments} give Table~\ref{tab:product-constants}.

\begin{table}[htbp]
\centering\small
\caption{Product-norm bounds for the value and normal parts. The columns multiply $\mathcal P_0$ and $\mathcal P_1$, respectively.}\label{tab:product-constants}
\begin{tabular}{@{}lcc@{}}\toprule
Quantity & Value Part & Normal Part\\\midrule
$\delta^{-4}\|\bar Z\|_\times^2$ & $13/280$ & $1/28$\\
$\delta^{-2}\|\nabla_\times\bar Z\|_\times^2$ & $13/70$ & $1/2$\\
$\|\bar Z_{rr}\|_\times^2$ & $3/2$ & $10\sqrt3$\\
$T_2^2$ & $6/5$ & $29/(14\sqrt3)$\\
$\|\Delta_\Gamma\bar Z\|_\times^2$ & $26/35$ & $1/(14\sqrt3)$\\\bottomrule
\end{tabular}
\end{table}

For the normal column use $q\leq\mu^2$ and $\mu/(q+4)\leq1/(2\sqrt3)$; in particular, the radial ratio is $60\mu/(q+4)\leq10\sqrt3$. For the value column compare coefficients of $16+4q+q^2$. The five rows are bounded by $M\mathcal B,G\mathcal B,R\mathcal B,T\mathcal B,W\mathcal B$, where
\begin{equation}\label{eq:H2-product-constants}
M=\frac{13}{280},\quad G=\frac12,\quad R=10\sqrt3,\quad T=\frac65,\quad W=\frac34
\end{equation}
On the positive half-tube, \eqref{eq:hermite-Gram} and coefficient comparison bound $\|Z_{rr}+\Delta_\Gamma Z\|_+^2$ by $\tfrac34\mathcal P_0$ for the value part and $\tfrac12(R+T)\mathcal P_1$ for the normal part. Weighted Cauchy--Schwarz between the parts and the mass moments therefore give
\begin{equation}\label{eq:H2-one-sided-bounds}
\|Z_{rr}+\Delta_\Gamma Z\|_+^2\leq V\mathcal B,
\quad V=\tfrac34+\tfrac{R+T}{2}\leq\tfrac{81}{8},
\qquad \delta^{-4}\|Z\|_+^2\leq\tfrac{69}{35}B_1
\end{equation}
The one-sided cross terms are included in $V$. Bochner's identity adds at most $\theta^2G\mathcal B$ to the tangential Hessian bound. These estimates prove convergence in $H^2$, the prescribed boundary traces, and support of $z$ in $\overline{D_\delta}$.

\subsection{Physical Norms}\label{section:physical-norms}
For the quadratic lifting, assume $\theta,\sigma\leq1/10$ and put $a=10/9$, $J=23/20$. The Jacobian is at most $J$, and $J_r\nu_i^2\leq11/9$ for $\nu_i=(1-r\kappa_i)^{-1}$. The $rr$ Hessian component is unchanged; mixed and tangential components have multipliers at most $a$ and $a^2$. The second-derivative Laplacian remainder acts only on the tangential Hessian, with coefficient at most $a^2-1$. Differentiating the normal coordinates bounds the first-derivative remainders in the Hessian and Laplacian by $1/(4\delta)$: their squared operator norms, multiplied by $\delta^2$, are at most $(a^2+2a^4+a^6)/100<1/16$ and $(a^2+a^6)/50<1/16$, respectively. The Hessian remainder has no $rr$ component. Thus, using the Bochner correction $\theta^2G\leq1/200$,
\begin{align}
\delta^{-2}\|\nabla\bar z\|_{\IR^d}^2
&\leq\tfrac{11}{9}GB_0=\tfrac{11}{18}B_0\label{eq:H2-physical-G}\\
\|D^2\bar z\|_{\IR^d}^2
&\leq J\left[R+\left(a^2\sqrt{T+\tfrac1{200}}+\tfrac{\sqrt G}{4}\right)^2\right]B_0
\leq23B_0\label{eq:H2-physical-H}\\
\|\Delta z\|_{D_\delta}^2
&\leq J\left[\sqrt V+(a^2-1)\sqrt{W+\tfrac1{200}}+\tfrac{\sqrt G}{4}\right]^2B_0
\leq\tfrac{73}{5}B_0\label{eq:H2-physical-L}
\end{align}
The mass bound \eqref{eq:H2-one-sided-bounds} and $J(69/35)<73/5$ give
\begin{equation}\label{eq:H2-physical-volume}
\|\Delta z\|_{D_\delta}^2+\tau^{-1}\delta^{-4}\|z\|_{D_\delta}^2
\leq\tfrac{73}{5}\|w\|_{b_h}^2
\end{equation}

\subsection{Face Traces and Coercivity}\label{section:quadratic-coercivity}
Apply the lattice trace estimate \eqref{eq:lattice-trace} to the coordinate derivatives of $\bar z$. Restriction to physical faces decreases the sums, so $h=\delta=\sqrt d\ell$, $d\leq3$, and \eqref{eq:H2-physical-G}--\eqref{eq:H2-physical-H} give
\begin{equation}\label{eq:H2-cartesian-faces}
\sum_{F\in\mcF_{h,\delta}}h^{-1}\|\partial_{n_F}z\|_{F\cap\Omega}^2
\leq\left(\tfrac{11\sqrt3}{18}+2\sqrt{\tfrac{253}{18}}\right)B_0
\leq\tfrac{43}{5}B_0
\end{equation}

\begin{proof}[Proof of Lemma~\ref{lem:quadratic-coercivity}]
For $C^1$ quadratics, tangential differentiation of the matching traces gives $\ldb\Delta v\rdb=\ldb\partial_{n_F}^2v\rdb$ and $\ldb\partial_{n_F}\Delta v\rdb=0$. Cellwise Green's formula, with $n_F$ outward from $T^+$, yields
\begin{equation}\label{eq:H2-Green}
Q(v,w)=(\Delta_h^2v,z)_{D_\delta}-(\Delta v,\Delta z)_{D_\delta}
+\sum_{F\in\mcF_{h,\delta}}(\ldb\partial_{n_F}^2v\rdb,\partial_{n_F}z)_{F\cap\Omega}
\end{equation}
Density extends the identity to the $H^2$ lifting, using \eqref{eq:lattice-trace}. By \eqref{eq:H2-physical-volume}, \eqref{eq:H2-cartesian-faces}, $\gamma_2\geq1$, and weighted Cauchy--Schwarz,
\begin{equation}\label{eq:H2-Q-bound}
|Q(v,w)|\leq\sqrt{\tfrac{116}{5}}\bigl(\|v\|_{a_h}^2+j_h(v,v)\bigr)^{1/2}\|w\|_{b_h}
\end{equation}
Young's inequality with parameter $4/5$ leaves coefficients $1/5$ and $\beta-29\geq1$ for $\|v\|_{a_h}^2+j_h(v,v)$ and $\|v\|_{b_h}^2$, respectively, proving the claim.
\end{proof}

\subsection{Error Estimates}\label{section:face-error-proof}
\begin{proof}[Proof of Theorem~\ref{thm:face-error}]
Use the consistency bounds \eqref{eq:cubic-consistency} and \eqref{eq:H2-Q-bound} for the respective degrees.
Extending the $H^4$ summand to the active mesh, these bounds give continuity and coercivity of $A_{h,j}$ on $H^4(\Omega)+V_h$. The $H^4$ summands have no jumps in derivatives through order three, so the same bounds apply. Consistency and Galerkin orthogonality therefore give
\begin{equation}\label{eq:H2-cea}
\tn u-u_h\tn_{h,j}\lesssim\inf_{v_h\in V_h}\tn u-v_h\tn_{h,j}
\end{equation}
Use the extension and spline quasi-interpolant from Section~\ref{section:interpolation}, now with $p=2,3$, and put $\widetilde\eta=Ev-\pi_hv$. The usual local estimates hold with approximation order $p+1$. For $p=2$, the fourth-derivative bound is $\|D^4\widetilde\eta\|_T\lesssim h^{-1}\|Ev\|_{H^4(\mcN_h(T))}$ for $h\leq1$, by local quadratic reproduction and scaling on full cells. In addition to the boundary traces in \eqref{eq:interpolation-boundary}, full-face trace estimates give
\begin{equation}\label{eq:H2-interpolation-faces}
j_h(v-\pi_hv,v-\pi_hv)
\lesssim\sum_{T\in\mcT_h(\Omega)}\sum_{m=p}^{p+1}
h^{2m-4}\|D^m\widetilde\eta\|_T^2
\end{equation}
The factor $h^{2p-3}$ in $j_h$ gives the weights $h^{2p-4}$ and $h^{2p-2}$ for derivatives of orders $p$ and $p+1$. Bounded overlap and extension stability thus imply
\begin{equation}\label{eq:H2-interpolation-energy}
\tn v-\pi_hv\tn_{h,j}\lesssim h^{p-1}\|v\|_{H^4(\Omega)},
\qquad v\in H^4(\Omega)
\end{equation}
This proves \eqref{eq:H2-error-energy}.

For $e=u-u_h$, let $z$ solve \eqref{eq:dual-problem}. Since $z\in H^4(\Omega)\cap H^2_0(\Omega)$, $j_h(e,z)=0$ and the identity \eqref{eq:dual-consistency} holds with the elementwise bi-Laplacian. As in the proof of Theorem~\ref{thm:error-estimates},
\begin{align}
\|e\|_\Omega^2
&=A_{h,j}(e,z-\pi_hz)
-\tau h^4(\Delta_h^2e,e)_{D_\delta}
\label{eq:H2-dual-identity}\\
&\lesssim(h^{p-1}+h^2)\tn e\tn_{h,j}\|e\|_\Omega
\label{eq:H2-dual-bound}
\end{align}
using \eqref{eq:H2-interpolation-energy}, \eqref{eq:dual-regularity}, and $\delta=h$. Division by $\|e\|_\Omega$, unless $e=0$, proves \eqref{eq:H2-error-L2}. For quadratics the dual approximation supplies one power of $h$; for cubics it supplies two.
\end{proof}

\section{Numerical Verification}\label{section:numerical-verification}
The following settings and checks document the experiments in Section~\ref{section:numerics}. Computations use double precision in MATLAB R2025b. The numerical values below are taken from the saved outputs of those experiments and their verification runs.

\subsection{Geometry, Quadrature, and Solvers}\label{section:verification-settings}
The physical strip is constructed from the exact ellipse parametrization $\gamma(t)=(a\cos t,b\sin t)$ as $x(t,r)=\gamma(t)-rn(t)$, $0<r<\delta$, with Jacobian $|\gamma'(t)|(1-r\kappa(t))$. In each cell, the normal segments are clipped to the cell boundaries. The angular intervals are split at crossings of the boundary and its inner parallel curve with cell faces and at changes of the clipping face through a cell corner. Gauss quadrature is then applied on each smooth integration piece. The standard rules use 12 points per direction for assembly and 18 for error evaluation. The refinement checks use 18 assembly points; the small-penalty check also increases the error rule to 24 points. All normal, curvature, and surface-derivative terms are evaluated from the exact geometry.

For each retained matrix, let $D_A=\operatorname{diag}(|A_{ii}|^{-1/2})$, $B=D_AAD_A$, and $r=D_A\mathbf b$. We solve $B\mathbf z=r$ by MATLAB's sparse direct solver and recover the coefficients as $D_A\mathbf z$. Thus there is no iterative linear-solver stopping tolerance. Across the 232 LS Nitsche systems in the convergence, cut-position, fixed-method, parameter, and quadratic studies, the maximum scaled relative residual $\|B\mathbf z-r\|_2/\|r\|_2$ is $8.03\cdot10^{-13}$ (rounded upward). The condition estimates reported below are \texttt{condest} estimates of $\kappa_1(B)$.

For $D_G=\operatorname{diag}(G_{ii}^{-1/2})$, put $C=D_GAD_G$ and $H=D_GGD_G$. The generalized eigenvalue bisection starts with $[-1,1]$, expands the bracket as necessary, and stops when its width is at most $10^{-12}\max(1,|\lambda_-|,|\lambda_+|)$. Positive definiteness is tested by Cholesky factorization after symmetric diagonal scaling in the energy metric. Independent checks use a dense symmetric eigenproblem after Cholesky whitening, and, for the fine fixed-method standard matrix, shift-invert iteration with tolerance $10^{-11}$ and at most 2000 iterations. For the scaled pair $(C,H)$, the reported eigenpair residual is
\begin{equation}\label{eq:verification-eigen-residual}
\frac{\|C\mathbf v-\lambda H\mathbf v\|_2}
{\|C\mathbf v\|_2+|\lambda|\|H\mathbf v\|_2}
\end{equation}
The criterion $\widehat\kappa_1(B)\epsilon_{\rm mach}>10^{-4}$ flags unresolved diagnostics; a small linear-solve residual alone is not used to certify an eigenvalue.

Basis functions are sorted by increasing $G_{ii}=\tn\varphi_i\tn_h^2$ (or the augmented norm for quadratics), and the largest initial group satisfying \eqref{eq:basis-removal-budget} is removed. The budget constant is $0.01$, with $p=4$ or $p=2$ as appropriate; it is zero in the fixed-method and parameter studies. Standard Nitsche and LS Nitsche use the same retained space. Table~\ref{tab:verification-spectra} includes the removal counts for the quartic convergence study.

\subsection{Quadrature and Spectral Checks}\label{section:verification-checks}
Table~\ref{tab:verification-checks} reports quadrature-refinement and independent spectral comparisons. The fixed-method checks use $\beta=30$, $\tau=0.01$, and $\delta=\sqrt2/4$, with the nested phases of Section~\ref{section:fixed-method} and standard penalty multiplier $1000$. The small-$\tau$ and small-$\beta$ checks use $\delta=h$ and phase $(0.23,0.37)$. Removal is disabled in these checks. Relative quadrature differences use the refined result in the denominator, except for the $\beta=3$ check, which uses the original result. Matrix differences use the Frobenius norm; eigenvalue differences are absolute.

\begin{table}[htbp]
\centering\footnotesize
\caption{Verification checks for quartic splines. Quadrature refinements increase assembly from 12 to 18 points per direction; for $\beta=3$, the error rule also increases from 18 to 24. The fixed-method dense comparison reports the larger discrepancy of the LS and standard eigenvalues. Maxima are taken over the listed mesh sizes.}\label{tab:verification-checks}
\setlength{\tabcolsep}{3pt}
\renewcommand{\arraystretch}{1.15}
\begin{tabular}{@{}p{0.29\textwidth}p{0.47\textwidth}r@{}}\toprule
Case & Quantity & Value\\\midrule
Fixed method, $N=16$ & Relative matrix change, $12\to18$ points & $1.46\cdot10^{-15}$\\
Fixed method, $N=16$ & Relative energy-error change & $1.01\cdot10^{-13}$\\
$\beta=30$, $\tau=10^{-6}$, $N=32$ & Relative matrix change, $12\to18$ points & $1.93\cdot10^{-15}$\\
$\beta=30$, $\tau=10^{-6}$, $N=32$ & Relative $L^2$-error change & $1.81\cdot10^{-7}$\\
$\beta=3$, $\tau=0.01$, $N\in\{8,32\}$ & Maximum relative $L^2$-error change & $1.15\cdot10^{-8}$\\
$\beta=3$, $\tau=0.01$, $N\in\{8,32\}$ & Maximum eigenvalue change & $2.78\cdot10^{-17}$\\
\midrule
Fixed method, $N=16$ & Bisection/dense eigenvalue difference & $9.31\cdot10^{-13}$\\
Fixed standard, $N=64$ & Bisection/shift-invert difference & $2.15\cdot10^{-13}$\\
Fixed standard, $N=64$ & Eigenpair residual \eqref{eq:verification-eigen-residual} & $1.25\cdot10^{-12}$\\
$\beta=3$, $\tau=0.01$, $N=8$ & Bisection/dense eigenvalue difference & $5.83\cdot10^{-14}$\\
$\beta=3$, $\tau=0.01$, $N=8$ & Eigenpair residual \eqref{eq:verification-eigen-residual} & $3.87\cdot10^{-14}$\\
\bottomrule\end{tabular}
\end{table}
\FloatBarrier

\subsection{Representative Stability Diagnostics}\label{section:verification-spectra}
Table~\ref{tab:verification-spectra} gives numerical values underlying the quartic convergence and fixed-method studies. Each LS/standard pair uses the same energy metric $G$. In the fixed-method study, the LS Nitsche generalized eigenvalue remains close to $0.90$ while the scaled condition estimate grows substantially. These quantities therefore measure distinct properties.

\begin{table}[htbp]
\centering\footnotesize
\caption{Representative quartic-spline diagnostics with $\beta=30$ and $\tau=0.01$. Each $\widehat\kappa_1$ estimates the condition number after symmetric diagonal scaling. In (a), the standard penalties match those of LS Nitsche and both methods use the same removed basis functions. In (b), removal is disabled and the standard penalties have the fixed calibration factor $1000$.}\label{tab:verification-spectra}
\setlength{\tabcolsep}{5pt}
\renewcommand{\arraystretch}{1.12}
\begin{tabular}{@{}rrrrrr@{}}\toprule
$N$ & Removed & $\lambda_{\min}^{\rm LS}$ & $\widehat\kappa_1^{\rm LS}$ & $\lambda_{\min}^{\rm Std}$ & $\widehat\kappa_1^{\rm Std}$\\\midrule
\multicolumn{6}{@{}l}{(a) Convergence study: $\delta=h$, phase $(0.23,0.37)$}\\[2pt]
8 & 8 & $0.899930$ & $2.07\cdot10^{3}$ & $-1.04\cdot10^{-2}$ & $1.05\cdot10^{6}$\\
16 & 9 & $0.894831$ & $1.24\cdot10^{3}$ & $-1.03\cdot10^{-2}$ & $1.62\cdot10^{7}$\\
32 & 5 & $0.892291$ & $1.76\cdot10^{4}$ & $-1.03\cdot10^{-2}$ & $8.48\cdot10^{6}$\\
64 & 13 & $0.890980$ & $2.76\cdot10^{5}$ & $-1.05\cdot10^{-2}$ & $1.10\cdot10^{8}$\\
\midrule
\multicolumn{6}{@{}l}{(b) Fixed-method study: $\delta=\sqrt2/4$, nested phases}\\[2pt]
8 & 0 & $0.900101$ & $1.58\cdot10^{3}$ & $1.42\cdot10^{-6}$ & $5.50\cdot10^{6}$\\
16 & 0 & $0.898106$ & $3.43\cdot10^{3}$ & $-6.13\cdot10^{-7}$ & $2.33\cdot10^{8}$\\
32 & 0 & $0.896314$ & $6.98\cdot10^{5}$ & $-7.54\cdot10^{-7}$ & $6.33\cdot10^{7}$\\
64 & 0 & $0.894657$ & $1.84\cdot10^{8}$ & $-8.12\cdot10^{-7}$ & $2.57\cdot10^{9}$\\
\bottomrule\end{tabular}
\end{table}

\FloatBarrier

\clearpage
\normalsize

\bigskip
\paragraph{Acknowledgments.}
This research was supported in part by the Swedish Research Council grants
Nos.\ 2021-04925, 2025-05562, and 2026-06012; the Swedish Research Programme Essence.

\Needspace{9\baselineskip}
\paragraph{Declaration of AI Use.}
The initial manuscript and code were developed without AI tools. The authors subsequently used OpenAI GPT-6 Astra to assist with manuscript revision and running the numerical experiments, and Anthropic Fable 5.1 to review the final draft. The authors reviewed the AI-assisted work, made revisions where necessary, and retain full responsibility for the article's content.

\bigskip
\bigskip
{ \footnotesize
\noindent
{\bf Authors' Addresses:}

\smallskip
\noindent
Mats G. Larson,  \quad \hfill \addressumushort\\
{\tt mats.larson@umu.se}

\smallskip
\noindent
Karl Larsson, \quad \hfill \addressumushort\\
{\tt karl.larsson@umu.se}

} 
\end{document}